\documentclass[11.5pt]{amsart}
\usepackage{geometry}
\usepackage{nicefrac}
\usepackage{graphicx}
\usepackage{amssymb}
\usepackage[colorlinks=true, citecolor=blue, linkcolor=blue, urlcolor=blue]{hyperref}
\newgeometry{asymmetric, centering, margin=1in}
\numberwithin{equation}{section}

\newtheorem{lemma}{Lemma}
\newtheorem{proposition}{Proposition}
\newtheorem{theorem}{Theorem}

\theoremstyle{remark}

\def\C{\mathbb C}
\def\R{\mathbb R}

\def\d{\partial}
\newcommand{\pair}[1]{\left\langle #1 \right\rangle}
\newcommand{\norm}[1]{\left\lVert #1 \right\rVert}
\newcommand{\abs}[1]{\left\lvert #1 \right\rvert}
\DeclareMathOperator{\Id}{Id}
\DeclareMathOperator{\End}{End}
\DeclareMathOperator{\tr}{tr}
\DeclareMathOperator{\Real}{Re}
\DeclareMathOperator{\Imag}{Im}

\date{Compiled \today}
\title[Recovery of Time-Dependent Coefficients]{Recovery of a Time-Dependent Hermitian Connection and Potential Appearing in the Dynamic Schr\"odinger Equation}
\author[A. Tetlow]{Alexander Tetlow}

\begin{document}

\begin{abstract}
We consider, on a trivial vector bundle over a Riemannian manifold with boundary, the inverse problem of uniquely recovering time- and space-dependent coefficients of the dynamic, vector-valued Schr\"odinger equation from the knowledge of the Dirichlet-to-Neumann map. We show that the D-to-N map uniquely determines both the connection form and the potential appearing in the Schr\"odinger equation, under the assumption that the manifold is either a) two-dimensional and simple, or b) of higher dimension with strictly convex boundary and admits a smooth, strictly convex function.
\end{abstract}
\maketitle

\section{Introduction}

\subsection{Statement of the Problem}
Let $T>0$ be fixed, and let $(M,g)$ be a connected, compact, smooth Riemannian manifold of dimension $m\geq2$ with boundary $\d M$. In what follows, we shall additionally assume that $(M,g)$ is non-trapping. Consider a trivial Hermitian vector bundle $E=M\times\C^n$ equipped with the Hermitian inner product $\pair{\cdot,\cdot}_E$.\\

We say that a connection $\nabla:C^\infty(M;E)\rightarrow C^\infty(M;E\otimes T^\ast M)$ is compatible with the Hermitian structure of $E$ if for any sections $u,v\in C^\infty(M;E)$ it holds that\begin{equation}\label{HIP}d\pair{u,v}_E=\pair{\nabla u,v}_E+\pair{u,\nabla v}_E,\end{equation}where both sides of the above are regarded as sections of the cotangent bundle.\\

Such a connection has the form $\nabla=d+A$, where $A=A_idx^i$ and each $A_i(x)$ is given by an $n\times n$ skew-Hermitian matrix. In what follows, we allow the connection form $A$ to also depend smoothly on time, and write $\nabla^A:C^\infty((0,T)\times M;E)\rightarrow C^\infty((0,T)\times M;E\otimes T^\ast M)$ for the time-dependent connection corresponding to the connection form $A$. In other words, $\nabla^{A(t)}$ is a connection on $C^\infty(M;E)$ for each $t\in[0,T]$, and each $\nabla^{A(t)}$ is compatible with the Hermitian metric on $E$.\\

We can define a natural $L^2$-inner product on $C^\infty((0,T)\times M;E)$ via \[\pair{u,v}_{L^2((0,T)\times M;E)}=\int_0^T\int_M\pair{u,v}_EdVdt,\]where $dV$ denotes the usual Riemannian volume measure of $(M,g)$. We can similarly define a natural $L^2$-inner product on $C^\infty((0,T)\times M;E\otimes T^\ast M)$. For $E$-valued $1$-forms $\alpha=\alpha_jdx^j$ and $\beta=\beta_jdx^j$, we set\[\pair{\alpha,\beta}_{L^2((0,T)\times M;E\otimes T^\ast M)}=\int_0^T\int_Mg^{ij}\pair{\alpha_i,\beta_j}_EdVdt,\]where $g^{ij}$ denotes the inverse of the metric tensor.\\

We let $(\nabla^A)^\ast$ denote the adjoint of $\nabla^A$ with respect to the above inner products. We can then define the connection Laplacian $\Delta_A=-(\nabla^A)^\ast\nabla^A$, which corresponds to the connection form $A$.\\

We can compute local expressions for $(\nabla^A)^\ast$ and $\Delta_A$. Consider a section $u\in C^\infty(M;E)$ and an $E$-valued $1$-form $\beta=\beta_jdx^j$ supported on a local trivialisation. Since $A$ is skew-Hermitian, it holds that\[\pair{Au,\beta}_{L^2((0,T)\times M;E\otimes T^\ast M)}=\int_0^T\int_Mg^{ij}\pair{A_iu,\beta_j}_EdVdt=-\int_0^T\int_M\pair{u,g^{ij}A_i\beta_j}_EdVdt.\]

Letting $(A,\beta)_g=g^{ij}A_i\beta_j$, we see that $(\nabla^A)^\ast=d^\ast-(A,\cdot)_g$. Therefore, we have\[\Delta_Au=-d^\ast du-d^\ast(Au)+(A,du)_g+(A,Au)_g.\]Recall that $d^\ast\alpha=-\abs{g}^{-\frac{1}{2}}\d_i(\abs{g}^{\frac{1}{2}}g^{ij}\alpha_j)$, where $\abs{g}$ is the determinant of the metric tensor. Hence it holds that $d^\ast(Au)=(d^\ast A)u-(A,du)_g$. Thus, we conclude that\begin{equation}\Delta_Au=-d^\ast du+2(A,du)_g-(d^\ast A)u+(A,Au)_g.\label{CLL}\end{equation}

Lastly, we say that a section $V\in C^\infty((0,T)\times M;\C^{n\times n})$ is a potential if $V$ is Hermitian or, equivalently, for any sections $u,v\in C^\infty((0,T)\times M;E)$ it holds that\[\pair{Vu,v}_E=\pair{u,Vv}_E.\]

Let $\Delta_A$ and $V$ be as above and consider the following initial and boundary value problem for sections $u\in C^\infty((0,T)\times M;E)$.\begin{equation}\begin{split}\label{1.1}i\d_tu(t,x)+\Delta_{A(t)}u(t,x)+V(t,x)u(t,x)&=0\textrm{ in }(0,T)\times M,\\u(t,x)&=f\textrm{ on }(0,T)\times\d M,\\u(0,x)&=0\textrm{ in }M,\end{split}\end{equation}where the inhomogeneous Dirichlet data is given by $f\in C^\infty((0,T)\times\d M;E)$ satisfying $f\vert_{t=0}=\d_tf\vert_{t=0}=0$. We can then define the associated Dirichlet-to-Neuman map via\[\Lambda_{A,V}f={\nabla^A_\nu u}\Big\vert_{(0,T)\times\d M}\ ,\]where $\nu$ denotes the outward pointing unit normal vector field on $\d M$.\\

There is a natural gauge group associated with the equation above. Let $G:(0,T)\times M\rightarrow U(n)$ be a smooth map such that $G(t)\vert_{\d M}=\Id$, and choose $A_2=G^{-1}A_1G+G^{-1}dG$ and $V_2=G^{-1}V_1G+iG^{-1}\d_tG$. It then holds that $\nabla^{A_2}=G^{-1}\nabla^{A_1}G$, and hence that $\Delta_{A_2}=G^{-1}\Delta_{A_1}G$. We observe that if $u$ solves (\ref{1.1}) with $A=A_2$ and $V=V_2$, then $Gu$ solves the equation with $A=A_1$ and $V=V_1$, since\[(i\d_t+\Delta_{A_1}+V_1)Gu=G(i\d_t+\Delta_{A_2}+V_2)u=0.\]Furthermore, we observe that when the pairs $(A_1,V_1)$ and $(A_2,V_2)$ are as above, it holds that $\Lambda_{A_1,V_1}=\Lambda_{A_2,V_2}$. Therefore, we can only hope to recover the pair $(A,V)$ up to a gauge transform. The aim of the present work is to establish unique recovery of the connection form and potential from the knowledge of the Dirichlet-to-Neumann map, modulo gauge invariance.\\

\subsection{History of the Problem}

Literature dealing with the recovery of space- and time-dependent potentials of the dynamic Schr\"odinger equation is limited, even in the scalar case. For Euclidean domains, it was shown in \cite{11} that the time-dependent electromagnetic potentials are uniquely determined by the Dirichlet-to-Neumann map. Logarithmic-stable determination was shown for the electric potential in \cite{9}, and this result was extended to the full electromagnetic potential in \cite{7}, provided that the time-independent part of the magnetic potential is sufficiently small. Indeed, it was only recently shown in \cite{22} that time-dependent electromagnetic potentials in a Euclidean domain can be H\"older-stably recovered from the knowledge of the D-to-N map. We also mention here the recent work of \cite{BbF}, which establishes logarithmic and double-logarithmic stability estimates for the same problem with partial data.\\

In the Riemannian setting, \cite{3} and \cite{4} establish, respectively, H\"older-stable recovery of a time-independent magnetic and electric potential of the dynamic Schr\"odinger equation on a simple manifold. These results were extended to simultaneous recovery of both electromagnetic potentials in \cite{2}. In the case of time-dependent potentials in the Riemannian context, the only result is that of \cite{KT}, establishing, on a simple manifold, the H\"older-stable recovery of both potentials from the knowledge of the Dirichlet-to-Neumann map.\\

In the case of the vector-valued dynamic Schr\"odinger equation, there are, to the best of the author's knowledge, no results establishing unique recovery even for time-independent coefficients. However, such results do exist for the related case of the stationary Schr\"odinger equation. In particular, for the stationary Schr\"odinger equation on a trivial vector bundle over a Euclidean domain, \cite{Esk1} establishes unique recovery of a connection form and potential from the knowledge of the Dirichlet-to-Neumann map. Additionally, for the stationary Schr\"odinger equation on a Hermitian vector bundle over a two-dimensional Riemann surface, \cite{AGTU} uniquely recovers the coefficients from Cauchy data at the boundary.\\

Let us also mention the paper \cite{MC1}, where it is conjectured that the Dirichlet-to-Neumann maps for two connection Laplacians coincide in the case of the stationary Schr\"odinger equation if and only if the associated connection forms are gauge equivalent. The present work solves this conjecture in the non-stationary case. More precisely, we show that the Dirichlet-to-Neumann map uniquely determines, up to gauge invariance, the space- and time-dependent connection form and potential appearing in the dynamic Schr\"odinger equation on a trivial vector bundle over a Riemannian manifold, provided that the manifold in question satisfies certain geometric conditions.\\

Finally, we note the works \cite{MC2}, \cite{KOP}, and \cite{CLOP}, where various inverse problems for partial differential equations involving connections are considered. In particular, \cite{MC2} establishes unique recovery of a time-independent unitary Yang-Mills connection on a Hermitian vector bundle, as well as recovering the bundle structure. Similarly, in the case of the wave equation, \cite{KOP} uses techniques from the boundary control method to reconstruct a Riemannian manifold and Hermitian vector bundle with a time-independent compatible connection from the knowledge of the associated hyperbolic Dirichlet-to-Neumann map, and the work \cite{CLOP} considers the non-linear inverse problem of recovering a time-independent connection from a cubic wave equation on a Hermitian vector bundle over the Minkowski space $\R^{1+3}$.\\

Lastly, let us mention the recent work of \cite{MV}, where the authors recover a time-dependent potential of the wave equation on a trivial vector bundle over a Euclidean domain from knowledge of the input-output operator on the partial boundary.\\

\subsection{Geodesics and Parallel Transport}

Let us assume that $(M,g)$ is non-trapping, which is to say that every geodesic in $M$ reaches the boundary in finite time. We now take a moment to recall certain key facts relating to the geodesics in $M$.\\

Given $x\in M$ and $\theta\in T_xM$, we denote by $\gamma_{x,\theta}$ the geodesic with initial point $x$ and initial direction $\theta$. We define the sphere bundle of $M$ via\[SM=\{(x,\theta)\in TM:\abs{\theta}_g=1\}.\]

Likewise, we define the submanifold of inner vectors $\d_+SM$ via\[\d_+SM=\{(x,\theta)\in SM:x\in\d M,\ \pair{\theta,\nu(x)}_{g(x)}<0\},\] where $\nu(x)$ is the outward pointing unit normal vector at $x\in\d M$, and we define the submanifold of outer vectors $\d_-SM$ via\[\d_-SM=\{(x,\theta)\in SM:x\in\d M,\ \pair{\theta,\nu(x)}_{g(x)}>0\}.\] Then, for $\gamma_{x,\theta}$ such that $(x,\theta)\in\d_+SM$, we can define the exit-time $\rho_+(x,\theta)$ of the geodesic in $M$ by\[\rho_+(x,\theta)=\min\{s>0:\gamma_{x,\theta}(s)\in\d M\}.\]

Given the above, we further recall the parallel transport equations associated to a connection $\nabla^A$. For any geodesic $\gamma_{x,\theta}$ with $(x,\theta)\in\d_+SM$, and any initial vector $w\in E_x$, we consider the parallel transport equation along $\gamma_{x,\theta}$, given by\begin{equation}\begin{split}\label{PT1}\Big[\d_r+A\big(\gamma'_{x,\theta}(r)\big)\Big]W=&\ 0\\W(0)=&\ w.\end{split}\end{equation}The transport of $w\in E_x$ along $\gamma_{x,\theta}$ is thus given by $W(r)$.\\

It is frequently helpful to consider the fundamental matrix solution $U_A:\big[0,\rho_+(x,\theta)\big]\rightarrow U(n)$ of the parallel transport equation:\begin{equation}\begin{split}\label{PT2}\big[\d_r+A\big(\gamma'_{x,\theta}(r)\big)\big]U_A=&\ 0\\U_A(0)=&\ \Id.\end{split}\end{equation}

It is clear from the above that the transport of $w\in E_x$ along $\gamma_{x,\theta}$ is given by $W(r)=U_A(r)\cdot w$.\\

Given $(x,\theta)\in\d_+SM$, we define the scattering data for the connection as the map $C_A:\d_+SM\rightarrow U(n)$ given by\begin{equation}C_A(x,\theta):=U_A\big(\rho_+(x,\theta)\big).\label{ScatteringData}\end{equation}

Since elements of the gauge group must satisfy $G(t)\vert_{\d M}=\Id$, we note that the scattering data $C_A$ is gauge invariant. We are now in a position to state the main results of the present work.\\

\subsection{Main Results}
\begin{theorem}\label{t1}Suppose that for $j=1,2$, we have connection forms $A_j\in C^\infty((0,T)\times M;\C^{n\times n}\otimes T^\ast M)$, and potentials $V_j\in C^\infty((0,T)\times M;\C^{n\times n})$. Then $\Lambda_{A_1,V_1}=\Lambda_{A_2,V_2}$ implies that $C_{A_1}=C_{A_2}$.
\end{theorem}

\begin{theorem}\label{t2}
Assume the conditions of Theorem \ref{t1} hold, and assume further that $M$ is either i) $2$-dimensional and simple, or ii) of dimension $m\geq3$ with strictly convex boundary, and admits a smooth strictly convex function. Then $(A_1,V_1)$ is gauge equivalent to $(A_2,V_2)$.
\end{theorem}

These results are, as far as the author is aware, the first dealing with the recovery of coefficients appearing in the dynamic Schr\"odinger equation on a vector bundle. In fact, the above results are the first showing recovery of time-dependent coefficients of any linear second-order partial differential equation with variable coefficients of leading order, in the vector-valued case. The proof of these results relies on the construction of Gaussian beam solutions which allow recovery of the scattering data corresponding to the connection form. This data is then used to recover the connection form and potential of the Schr\"odinger equation via the inversion of attenuated ray-transforms with matrix-weights corresponding to the connection form and potential we wish to recover. This last step relies on the results of \cite{C3} and \cite {C1}, which guarantee that the appropriate attenuated ray-transform is invertible when the base manifold is either i) two-dimensional and simple or ii) of higher dimension with strictly convex boundary and admits a smooth strictly convex function.\\

Here follows an outline of the present work. In section 2, we give some regularity results for the forward problem and for the Neumann trace. In section 3 we construct special Gaussian beam solutions for the Schr\"odinger equation. The proofs of Theorems \ref{t1} and \ref{t2} are given in sections 4 and 5 respectively.\\

\section{The Forward Problem}

Before we proceed, let us define for all $r,s\in(0,\infty)$ and for $X=M$ or $X=\d M$ the energy spaces $H^{r,s}((0,T)\times X;E)=H^r(0,T;L^2(X;E))\cap L^2(0,T;H^s(X;E))$, together with the associated norm\[\norm{u}^2_{H^{r,s}((0,T)\times X;E)}=\norm{u}^2_{H^r(0,T;L^2(X;E))}+\norm{u}^2_{L^2(0,T;H^s(X;E))}.\]

The main result of this section is the following well-posedness result for the Dirichlet-to-Neumann map:

\begin{proposition}\label{edit_p1}
The IBVP (\ref{1.1}) has a unique solution $u\in C^\infty((0,T)\times M;E)$. Further, there exists a constant $C>0$ such that the associated Dirichlet-to-Neumann map satisfies the estimate\begin{equation}\label{edit_DN}\norm{\Lambda_{A,V}f}_{L^2((0,T)\times\d M;E)}\leq C\norm{f}_{H^{\frac{9}{4},\frac{3}{2}}((0,T)\times\d M;E)}.\end{equation}
\end{proposition}

The proof of the above result essentially reduces to proving a suitable energy estimate for the source problem for the Schr\"odinger equation. Thus, for $F\in C^\infty((0,T)\times M;E)$ we consider the solution of the source problem\begin{equation}\begin{split}\label{s2.1}\big(i\d_t+\Delta_A+V\big)u&=F(t,x)\textrm{ in }(0,T)\times M,\\u(t,x)&=0\textrm{ on }(0,T)\times\d M,\\u(0,x)&=0\textrm{ in }M.\end{split}\end{equation}

\begin{proposition}\label{edit_p2}
The source problem (\ref{s2.1}) satisfies the energy estimates\begin{align*}\norm{u}_{L^\infty(0,T;L^2(M;E))}\leq \norm{F}_{L^2((0,T)\times M;E)} \\ \norm{u}_{L^\infty(0,T;H^1(M;E))}\leq \norm{F}_{H^{1,0}((0,T)\times M;E)} \\ \norm{\d_tu}_{L^\infty(0,T;L^2(M;E))}\leq \norm{F}_{H^{1,0}((0,T)\times M;E)} \\ \norm{u}_{H^{1,2}((0,T)\times M;E)}\leq C\norm{F}_{H^{1,0}((0,T)\times M;E)}.\end{align*}
\end{proposition}

\begin{proof}
By taking the inner product of (\ref{s2.1}) with $u$ and integrating by parts, we deduce that\begin{equation}\label{s2.2}i\int_M\pair{\d_tu,u}_EdV_g-\int_M\pair{\nabla^Au,\nabla^Au}_EdV_g+\int_M\pair{Vu,u}_EdV_g=\int_M\pair{F,u}_EdV_g\end{equation}Taking the imaginary part of (\ref{s2.2}) yields\[\frac{d}{dt}\Big(\norm{u(t)}_{L^2(M;E)}^2\Big)\leq C\Big(\norm{F(t,\cdot)}_{L^2(M;E)}\norm{u(t)}_{L^2{(M;E)}}+\norm{u(t)}_{L^2(M;E)}^2\Big).\]Then Gr\"onwall's inequality tells us that\begin{equation}\label{s2.3}\norm{u}_{L^\infty(0,T;L^2(M;E))}\leq C\norm{F}_{L^2((0,T)\times M;E)}.\end{equation}

On the other hand, taking the inner product of (\ref{s2.1}) with $\d_tu$, we can integrate by parts to deduce that\begin{equation}\label{s2.4}i\int_M\pair{\d_tu,\d_tu}_EdV_g-\int_M\pair{\nabla^Au,\nabla^A\d_tu}_EdV_g+\int_M\pair{Vu,\d_tu}_EdV_g=\int_M\pair{F,\d_tu}_EdV_g.\end{equation}Then, by setting\begin{gather*}\alpha(t;u,v)=\int_M\pair{\nabla^Au,\nabla^Av}_EdV_g-\int_M\pair{Vu,v}_EdV_g\end{gather*}and\begin{gather*}\alpha'(t;u,u)=\int_M\pair{(\d_tA)u,\nabla^Au}_EdV_g+\int_M\pair{\nabla^Au,(\d_tA)u}_EdV_g-\int_M\pair{(\d_tV)u,u}_EdV_g,\end{gather*}we can take the real part of (\ref{s2.4}) to conclude that\[\frac{d}{ds}\alpha(s;u,u)=\alpha'(s;u,u)-2\,\Real\int_M\pair{F(s,\cdot),\d_tu(s)}_EdV_g.\]

By integrating, we can rewrite the above as\begin{equation}\label{s2.5}\alpha(t;u,u)=\int_0^t\alpha'(s;u,u)ds-2\,\Real\int_0^t\pair{F(s,\cdot),\d_tu(s)}_{L^2(M;E)}ds,\end{equation}and since $\int_0^t\pair{F(s,\cdot),\d_tu(s)}_{L^2(M;E)}ds=\pair{F(t,\cdot),u(t)}_{L^2(M;E)}-\int_0^t\pair{\d_tF(s,\cdot),u(s)}_{L^2(M;E)}ds$, we can rewrite the identity (\ref{s2.5}) in the form\begin{equation}\label{s2.6}\alpha(t;u,u)\leq\int_0^t\alpha'(s;u,u)ds+2\norm{F(t)}_{L^2(M;E)}\norm{u(t)}_{L^2(M;E)}+2\int_0^t\norm{\d_tF(s,\cdot)}_{L^2(M;E)}\norm{u(s)}_{L^2(M;E)}ds.\end{equation}

Further, by expanding the inner products appearing in $\alpha(t;u,u)$ and using Cauchy's inequality with $\varepsilon=\frac{1}{2}$, we can deduce that\[\begin{split}\frac{1}{2}\norm{\nabla u(t)}_{L^2(M;E\otimes T^\ast M)}^2\leq\alpha(t;u,u)+\norm{V(t)}_{L^\infty(\C^{n \times n}\otimes T^\ast M)}\norm{u(t)}_{L^2(M;E)}^2\\+\norm{A(t)}_{L^\infty(\C^{n \times n}\otimes T^\ast M)}^2\norm{u(t)}_{L^2(M;E)}^2\end{split}\]whence\[\alpha(t;u,u)+\lambda\norm{u(t)}_{L^2(M;E)}^2\geq\frac{1}{2}\norm{u(t)}_{H^1(M;E)}^2,\] for $\lambda=\frac{1}{2}+\norm{V(t)}_{L^\infty(M;\C^{n\times n}\otimes T^\ast M)}+\norm{A(t)}_{L^\infty(M;\C^{n\times n}\otimes T^\ast M)}^2$. Then, combining the above with identity (\ref{s2.6}) and the definition of $\alpha'(t;u,v)$, we may deduce that $\norm{u(t)}_{H^1(M;E)}^2$ is bounded above by\begin{equation}\label{edit_2.8}2\lambda\norm{u(t)}_{L^2(M;E)}^2+4\norm{F(t)}_{L^2(M;E)}\norm{u(t)}_{L^2(M;E)}+C\Big(\int_0^t\norm{u(s)}_{H^1(M;E)}^2+\norm{\d_tF(s,\cdot)}_{L^2(M;E)}^2ds\Big).\end{equation}Let us briefly recall the inequality\[\norm{F(t,\cdot)}_{L^2(M;E)}^2=2\Real\int_0^t\pair{F(s,\cdot),\d_tF(s,\cdot)}_{L^2(M;E)}ds\leq\int_0^t\Big(\norm{F(s,\cdot)}_{L^2(M;E)}^2+\norm{\d_tF(s,\cdot)}_{L^2(M;E)}^2\Big)ds.\] Using the above in (\ref{edit_2.8}), together with Cauchy's inequality and the estimate (\ref{s2.3}), we deduce that\begin{equation}\label{s2.7}\norm{u(t)}_{H^1(M;E)}^2\leq C\Big(\int_0^t\norm{u(s)}_{H^1(M;E)}^2ds+\norm{F}_{H^{1,0}((0,T)\times M;E)}^2\Big).\end{equation}Then an application of Gr\"onwall's inequality tells us that\begin{equation}\label{s2.8}\norm{u}_{L^\infty(0,T;H^1(M;E))}\leq C\norm{F}_{H^{1,0}((0,T)\times M;E)}.\end{equation}

For the next estimate, we begin by applying $\d_t$ to (\ref{s2.1}). Using the expression (\ref{CLL}) for the connection Laplacian, we deduce that\begin{equation}\label{s2.9}(i\d_t+\Delta_A+V)\d_tu=\d_tF-2(\d_tA,du)_g+(\d_td^\ast A)u-(\d_tA,Au)_g-(A,(\d_tA)u)_g-(\d_tV)u.\end{equation}We now apply the estimate (\ref{s2.3}) to $\d_tu$, replacing $F$ appearing in (\ref{s2.3}) by the right-hand side of (\ref{s2.9}). We deduce, therefore, that\[\begin{split}&\norm{\d_tu}_{L^\infty(0,T;L^2(M;E))}\\\leq& C\norm{\d_tF-2(\d_tA,du)_g+(\d_td^\ast A)u-(\d_tA,Au)_g-(A,(\d_tA)u)_g-(\d_tV)u}_{L^2((0,T)\times M;E)}.\end{split}\]Using the estimates (\ref{s2.3}) and (\ref{s2.8}), on the right-hand side of the above, we observe that\begin{equation}\label{s2.10}\norm{\d_tu}_{L^\infty(0,T;L^2(M;E))}\leq C\norm{F}_{H^{1,0}((0,T)\times M;E).}\end{equation}

Lastly, we rearrange (\ref{s2.1}) to obtain \begin{equation}\begin{split}(\Delta_A+V)u=&\ F-i\d_tu\textrm{ in }(0,T)\times M\\u=&\ 0\textrm{ on }(0,T)\times\d M.\end{split}\end{equation}Then the bounds (\ref{s2.3}), (\ref{s2.8}) and (\ref{s2.10}) immediately imply the desired energy estimate\begin{equation}\label{NRGest}\norm{u}_{H^{1,2}((0,T)\times M;E)}\leq C\norm{F}_{H^{1,0}((0,T)\times M;E).}\end{equation}\end{proof}

We now turn to the proof of Proposition \ref{edit_p1}.

\begin{proof}[Proof of Proposition \ref{edit_p1}]

The unique solvability of the IBVP (\ref{1.1}) can be established in a similar manner to the scalar valued case, using the energy estimates of Proposition \ref{edit_p2}. Existence and uniqueness can then be proven using, for example, the Galerkin approach (see e.g. \cite[Theorem 2.3]{22} or \cite[Section 3, Theorem 10.1]{23}).\\

We now turn to establishing the bound (\ref{edit_DN}) for the Dirichlet-to-Neumann map. Recall the initial and boundary value problem (\ref{1.1}):\[\begin{split}(i\d_t+\Delta_{A}+V)u&=0\textrm{ in }(0,T)\times M,\\u(t,x)&=f\textrm{ on }(0,T)\times\d M,\\u(0,x)&=0\textrm{ in }M,\end{split}\]where the inhomogeneous Dirichlet data is given by $f\in C^\infty((0,T)\times\d M;E)$ satisfying $f\vert_{t=0}=\d_tf\vert_{t=0}=0$.\\

Note that we can find $\Phi\in C^\infty((0,T)\times M;E)$ such that\[\Phi(0,\cdot)=\d_t\Phi(0,\cdot)=0\textrm{ in M},\qquad\Phi=f\textrm{ on }\d M,\]and\begin{equation}\norm{\Phi}_{H^{3,2}((0,T)\times M;E)}\leq C\norm{f}_{H^{\frac{9}{4},\frac{3}{2}}((0,T)\times \d M;E)},\end{equation} for some $C>0$, depending only on $M$ and $T$. See \cite[Chapter 4, Section 2]{23} for a proof of this fact in the scalar case. The proof for vectors is analogous and, therefore, omitted. From the above, it holds that\begin{equation}\label{s2.14}F:=-(i\d_t+\Delta_A+V)\Phi\end{equation}satisfies $F(0,\cdot)=0$ in $M$. Then, letting $v$ be the solution of (\ref{s2.1}) corresponding to the source term $F$ defined in (\ref{s2.14}), we see that $u=\Phi+v$ is a solution to (\ref{1.1}). Then, it follows by (\ref{NRGest}) that\[\norm{u}_{H^{1,2}((0,T)\times M;E)}\leq C\norm{f}_{H^{\frac{9}{4},\frac{3}{2}}((0,T)\times \d M;E)}\]and applying this estimate with $f=0$ implies that such a solution $u$ is unique. Finally, we observe that\[\norm{\Lambda_{A,V}f}_{L^2((0,T)\times\d M;E)}\leq\norm{u}_{H^{1,2}((0,T)\times M;E)}\leq C\norm{\Phi}_{H^{3,2}((0,T)\times M;E)}\leq C\norm{f}_{H^{\frac{9}{4},\frac{3}{2}}((0,T)\times \d M;E)}.\]\end{proof}

\section{Construction of Gaussian Beam Solutions}
In this section, we shall construct Gaussian beam solutions to the Schr\"odinger equation which concentrate along geodesics in the high frequency limit.\\

Let $(\widehat{M},g)$ be a closed manifold. Recall that for a geodesic segment $\gamma:(a,b)\rightarrow\widehat{M}$ with no closed loops, there exist only finitely many values of $r\in(a,b)$ for which $\gamma$ self-intersects at $\gamma(r)$. We begin by recording the following system of Fermi coordinates near a geodesic, which we shall later use to construct our Gaussian beam solutions.\\

\begin{nobreak}\begin{lemma}\label{Fermi}
Let $(\widehat{M},g)$ be a compact $m$-dimensional manifold without boundary, $m\geq2$, and assume that $\gamma:(a,b)\rightarrow\widehat{M}$ is a unit-speed geodesic with no closed loops. Given a closed sub-interval $[a_0,b_0]$ of $(a,b)$ such that $\gamma\vert_{[a_0,b_0]}$ self-intersects only at $\gamma(r_j)$ with $a_0<r_1<\cdots<r_K<b_0$, and setting $r_0=a_0$, $r_{N+1}=b_0$, there exists an open cover $\{U_j,\phi_j\}_{j=0}^{K+1}$ of $\gamma([a_0,b_0])$ consisting of coordinate neighbourhoods with the following properties:
\begin{itemize}
\item $\phi_j(U_j)=I_j\times B$, where $I_j$ are open intervals and $B=B(0,\delta')$ is an open ball in $\R^{m-1}$, where $\delta'$ can be taken arbitrarily small.
\item $\phi_j(\gamma(r))=(r,0)$ for $r\in I_j$
\item $r_j$ only belongs to $I_j$ and $\overline{I_j}\cap\overline{I_k}=\emptyset,$ unless $\abs{j-k}\leq1$
\item $\phi_j=\phi_k$ on $\phi_j^{-1}((I_j\cap I_k)\times B)$
\end{itemize}
Furthermore, the metric in these coordinates satisfies $g^{jk}\vert_{\gamma(r)}=\delta^{jk}$ and $\d_ig^{jk}\vert_{\gamma(r)}=0$.
\end{lemma}\end{nobreak}

\begin{proof}
See e.g. \cite[Lemma 3.5]{C4} for details.
\end{proof}

We now turn to the construction of the Gaussian beam solutions. We consider here a non-tangential unit-speed geodesic in $M$ given by $\gamma:[0,L]\rightarrow M$. That is, $\gamma'(0)$ and $\gamma'(L)$ are both non-tangential to $\d M$, and $\gamma(r)\in M^{int}$ for $0<r<L$. Note that this implies that the geodesic $\gamma$ is not a closed loop in $M$.\\

We may then embed $(M,g)$ in some closed manifold $\widehat{M}$, and extend $\gamma$ to $\widehat{M}$ as a unit-speed geodesic $\gamma:[-\varepsilon,L+\varepsilon]\rightarrow\widehat{M}$. Our aim is to construct a Gaussian beam solution near $\gamma([0,L])$. We fix a point $x_0$ on $\gamma$, and apply Lemma \ref{Fermi} on $\widehat{M}$ with $a_0<0$ and $b_0>L$ chosen so that $\gamma(a_0)$ and $\gamma(b_0)$ are in the interior of $\widehat{M}\setminus M$. This gives us a system of coordinates $(r,y)$ around $x_0=(r_0,0)$, defined in a set $U=\{(r,y):\abs{r-r_0}<\delta,\ \abs{y}<\delta'\}$ such that the geodesic near $x_0$ is given by $\Gamma=\{(r,0):\abs{r-r_0}<\delta\}$.\\

The main aim of the present section is to establish the following result.

\begin{proposition}\label{edit_p3}
Let $(M,g)$ be non-trapping, let $\gamma:[0,L]\rightarrow M$ be a non-tangential geodesic, and let $s\gg 1$. There exists a function $v\in C^\infty((0,T)\times M;E)$, supported in a tubular neighbourhood of $\gamma$, which is an approximate solution of the Schr\"odinger equation in the sense that\begin{equation}\label{edit_p3est}\norm{(i\d_t+\Delta_{A}+V)v}_{L^2((0,T)\times M;E)}=O(s^{-1}),\qquad\norm{v}_{L^2((0,T)\times M;E)}=O(1).\end{equation}Away from the self-intersections of $\gamma$, this approximate solution is given by $v(t,r,y)=e^{is(\Psi(r,y)-st)}a(t,r,y)$, where the phase function $\Psi$ has the form $\Psi(r,y)=r+\frac{1}{2}H(r)y\cdot y+O(\abs{y}^3)$ and the amplitude satisfies\[a(t,r,0)=c_0\tilde{\chi}(t)e^{-\frac{1}{2}\int_{r_0}^r\tr H(\tilde{r})d\tilde{r}}\ U_{A}w+O(s^{-1})\]for some choices of arbitrary constant $c_0$, compactly supported smooth function $\tilde{\chi}$ and initial vector $w$.
\end{proposition}

\begin{proof}

The first step is the local construction of such an approximate solution. That is, we wish to construct a solution $v$ of the Schr\"odinger equation in $U$, with the form\[v=e^{is(\Psi(r,y)-st)}a(s;t,r,y),\]where $\Psi\in C^\infty(M;\C)$, $a\in C^\infty((0,T)\times M;E)$ are given near $\Gamma$, with $a$ supported in $\{\abs{y}<\delta'/2\}$. For convenience, we shall supress the dependence on $s$ of the amplitude function $a$.\\

In practice, we will determine the functions $\Psi,a$ by solving certain Eikonal and transport equations up to $N$th order on $\Gamma$, much as one would in the analogous construction for the stationary Schr\"odinger equation (see e.g. \cite[Theorem 5.4]{MC1}). As a result of using the phase $e^{is(\Psi-st)}$, we obtain the same Eikonal equation for $\Psi$ as in \cite{MC1}, but derive easier transport equations for the amplitude $a$ (compared to \cite[Theorem 5.4]{MC1}, where $a$ satisfies transport equations of $\overline{\d}$-type).\\

Therefore, let us begin by deriving these Eikonal and transport equations. Recalling the expression (\ref{CLL}) for the connection Laplacian, we first compute the Schr\"odinger operator applied to $v$:\begin{equation}\begin{split}\label{WKB}(i\d_t+\Delta_A+V)v=&e^{is(\Psi-st)}(i\d_t+\Delta_A+V)a+s^2e^{is(\Psi-st)}\Big(1-(d\Psi,d\Psi)_g\Big)a\\&+2ise^{is(\Psi-st)}\Big((d\Psi,\nabla^Aa)_g+\frac{1}{2}(\Delta_g\Psi)a\Big).\end{split}\end{equation}

In light of the above, we seek $\Psi$ satisfying the Eikonal equation\begin{equation}\label{Eikonal}(d\Psi,d\Psi)_g-1=0\textrm{ to $N$th order on }\Gamma,\end{equation}exactly as in the case of the stationary Schr\"odinger equation. On the other hand, the amplitude $a$ should satisfy, up to a small error, the transport equation\[s\Big((d\Psi,\nabla^Aa)_g+\frac{1}{2}(\Delta_g\Psi)a\Big)-\frac{i}{2}(i\d_t+\Delta_A+V)a=0\quad\textrm{to $N$th order on }\Gamma.\]\\

Thus, we begin by seeking a solution $\Psi$ of (\ref{Eikonal}) having the form $\Psi=\sum_{j=0}^{N}\Psi_j$, where\[\Psi_j(r,y)=\sum_{\abs{\alpha}=j}\frac{\Psi_{j,\alpha}(r)}{\alpha!}y^\alpha.\]Let us also write the metric in the form $g^{jk}=\sum_{l=0}^Ng^{jk}_l+r^{jk}_{N+1}$, where\[g^{jk}_l(r,y)=\sum_{\abs{\beta}=l}\frac{g^{jk}_{l,\beta}(r)}{\beta!}y^\beta,\quad r^{jk}_{N+1}=O(\abs{y}^{N+1}).\]

By the properties of the Fermi coordinates, we observe that $g^{jk}_0=\delta^{jk}$ and $g^{jk}_1=0$. Thus, we can immediately choose $\Psi_0(r)=r$ and $\Psi_1(r,y)=0$. Then, for $j,k=1...m$ and $\alpha,\beta=2...m$, we have\begin{equation}\begin{split}\label{G2}g^{jk}\d_j\Psi\d_k\Psi-1&=(1+g^{11}_2+\cdots)(1+\d_r\Psi_2+\cdots)(1+\d_r\Psi_2+\cdots)\\&+2(g^{1\alpha}_2+\cdots)(1+\d_r\Psi_2+\cdots)(\d_{y^\alpha}\Psi_2+\cdots)\\&+(\delta^{\alpha\beta}+g^{\alpha\beta}_2+\cdots)(\d_{y^\alpha}\Psi_2+\d_{y^\alpha}\Psi_3+\cdots)(\d_{y^\beta}\Psi_2+\d_{y^\beta}\Psi_3+\cdots)-1\\&=[2\d_r\Psi_2+\nabla_y\Psi_2\cdot\nabla_y\Psi_2+g^{11}_2]\\&+\sum_{p=3}^N\Big[2\d_r\Psi_p+2\nabla_y\Psi_2\cdot\nabla_y\Psi_p+\sum_{l=0}^pg^{11}_l \sum_{\substack{{j+k=p-l}\\{0\leq j,k<p}}}\d_r\Psi_j\d_r\Psi_k\\&+2\sum_{l=2}^pg^{1\alpha}_l\sum_{\substack{{j+k=p+1-l}\\{2\leq k<p}\\{0\leq j<p}}}\d_r\Psi_j\d_{y^\alpha}\Psi_k+\sum_{l=0}^{p-2}g^{\alpha\beta}_l\sum_{\substack{{j+k=p+2-l}\\{2\leq j,k<p}}}\d_{y^\alpha}\Psi_j\d_{y^\beta}\Psi_k\Big]+O\big(\abs{y}^{N+1}\big).\end{split}\end{equation}
In the last equality, we have chosen to collect the terms into homogeneous polynomials in $y$ (so that the first term is the second degree part of the right-hand side, and the rest are the parts of degree $p=3,\dots,N$). We first choose $\Psi_2$ such that the second-degree term $[2\d_r\Psi_2+\nabla_y\Psi_2\cdot\nabla_y\Psi_2+g^{11}_2]$ vanishes.\\

To this end, we choose $\Psi_2(r,y)=\frac{1}{2}H(r)y\cdot y$, where $H$ is a smooth, symmetric, complex matrix solving the matrix Riccati equation\begin{equation}\label{Riccati}H'(r)+H^2(r)=F(r),\end{equation}and $F(r)$ is the symmetric matrix such that $g^{11}_2(r,y)=-F(r)y\cdot y$. If we impose some initial condition $H(r_0)=H_0$ on this equation, where $H_0$ is  chosen to be a complex symmetric matrix with $\Imag(H_0)$ positive definite, then \cite[Lemma 2.56]{C5} implies that the matrix Riccati equation above has a unique smooth symmetric solution $H(r)$, for which $\Imag(H(r))$ is positive definite.\\

We now choose $\Psi_3$ so that the term corresponding to $p=3$ in the right-hand side of (\ref{G2}) vanishes. We obtain the equation\[2\d_r\Psi_3+2\nabla_y\Psi_2\cdot\nabla_y\Psi_3=F(r,y),\]where $F$ is a third-order polynomial in $y$ which only depends on $\Psi_2$ and $g$. This gives us a linear system of first-order ODEs for the Taylor coefficients $\Psi_{3,\alpha}(r)$, which can be solved uniquely if we prescribe some initial conditions at $r_0$. We may, then, repeat this argument in order to obtain $\Psi_4,\dots,\Psi_N$ by solving ODEs on $\Gamma$, given inital conditions at $r_0$.\\

Thus, we have $\Psi(r,y)=r+\frac{1}{2}H(r)y\cdot y+\tilde{\Psi}$, where $\tilde{\Psi}=O(\abs{y}^3)$. We now turn to finding the amplitude $a$ such that, up to a small error, we have\[s\Big((d\Psi,\nabla^Aa)_g+\frac{1}{2}(\Delta_g\Psi)a\Big)-\frac{i}{2}(i\d_t+\Delta_A+V)a=0\quad\textrm{to $N$th order on }\Gamma.\]

We choose $a$ of the form\[a=s^{\frac{m-1}{4}}(a_0+s^{-1}a_{1}+\cdots+s^{-N}a_{N})\chi(y/\delta'),\]where $\chi$ is a smooth function such that $\chi=1$ for $\abs{y}\leq\nicefrac{1}{4}$ and $\chi=0$ for $\abs{y}\geq\nicefrac{1}{2}$. Letting $\eta=\Delta_g\Psi$, it is enough to find $a_j$ such that\begin{equation}\begin{split}\label{transport}(d\Psi,\nabla^Aa_0)_g+\frac{1}{2}\eta a_0&=0\quad\textrm{to $N$th order on $\Gamma$}\\(d\Psi,\nabla^Aa_1)_g+\frac{1}{2}\eta a_1-\frac{i}{2}(i\d_t+\Delta_A+V)a_0&=0\quad\textrm{to $N$th order on $\Gamma$}\\\vdots\\(d\Psi,\nabla^Aa_N)_g+\frac{1}{2}\eta a_N-\frac{i}{2}(i\d_t+\Delta_A+V)a_{N-1}&=0\quad\textrm{to $N$th order on $\Gamma$}.\end{split}\end{equation}

We can write $\eta=\sum_{l=0}^N\eta_l+r_{N+1}$ and $a_0=a_{00}+\cdots+a_{0N}$, where each $\eta_l$, $a_{0l}$ is a homogenenous polynomial of order $j$ in $y$, and the remainder $r_{N+1}$ is $O(\abs{y}^{N+1})$. Thus, writing $A=A(\gamma')dr+A(\d_{y^\alpha})dy^\alpha$, we can rewrite the transport equation for $a_0$ in the system (\ref{transport}) above as\begin{equation}\begin{split}\label{edit_transport}\Big(1+g^{11}_2+\cdots\Big)\Big(1+\d_r\Psi_2+\cdots\Big)\Big(\d_r+A(\gamma')\Big)\Big(a_{00}+a_{01}+\cdots\Big)\\+\Big(g^{1\alpha}_2+\cdots\Big)\Big(1+\d_r\Psi_2+\cdots\Big)\Big(\d_{y^\alpha}+A(\d_{y^\alpha})\Big)\Big(a_{00}+a_{01}+\cdots\Big)\\+\Big(g^{\alpha1}_2+\cdots\Big)\Big(1+\d_{y^\alpha}\Psi_2+\cdots\Big)\Big(\d_r+A(\gamma')\Big)\Big(a_{00}+a_{01}+\cdots\Big)\\+\Big(\delta^{\alpha\beta}+g^{\alpha\beta}_2+\cdots\Big)\Big(\d_{y^\alpha}\Psi_2+\d_{y^\alpha}\Psi_3+\cdots\Big)\Big(\d_{y^\beta}+A(\d_{y^\beta})\Big)\Big(a_{00}+a_{01}+\cdots\Big)\\+\frac{1}{2}(\eta_0+\eta_1+\cdots)(a_{00}+a_{01}+\cdots)\\=0.\end{split}\end{equation}

We can then write $A(\theta)=\sum_{l=0}^NA_l(\theta)+R_{N+1}(\theta)$, where the entries of each $A_l(\theta)$ are homogeneous polynomials in $y$ of order $l$, and the remainder $R_{N+1}(\theta)$ is $O(\abs{y}^{N+1})$. As a result, we deduce that the left-hand side of (\ref{edit_transport}) becomes

\[\begin{split}\d_ra_{00}+A_0(\gamma')a_{00}+\frac{1}{2}\eta_0a_{00}\\+\d_ra_{01}+A_0(\gamma')a_{01}+A_1(\gamma')a_{00}+\nabla_y\Psi_2\cdot\nabla_ya_{01}+\nabla_y\cdot A_0(\gamma')a_{00}+\frac{1}{2}\eta_0a_{01}+\frac{1}{2}\eta_1a_{00}\\+\cdots\end{split}\]

We wish to find $a_{00}$ such that the first line of the above expression vanishes. To this end, we note that $\eta_0(r)=\Delta_g\Psi(r,0)=g^{\alpha\beta}\d_{y^\alpha}[H_{\alpha\beta}y^\alpha]=\tr H(r)$. We therefore choose $a_{00}$ such that\[\d_ra_{00}+A_0\big(\gamma'(r)\big)a_{00}+\frac{1}{2}\big(\tr H(r)\big)a_{00}=0.\]Since $A_0(t,r)=A(t,r,0)$, this equation has the solution\[a_{00}(t,r)=c_0\tilde{\chi}(t)e^{-\frac{1}{2}\int_{r_0}^r\tr H(\tilde{r})d\tilde{r}}\cdot U_Aw,\]where $\tilde{\chi}\in C_0^\infty((\tau,T-\tau))$ satisfies $\tilde{\chi}=1$ on $[2\tau, T-2\tau]$, $0\leq\tilde{\chi}\leq1$ and $\norm{\tilde{\chi}}_{W^{k,\infty}(\R)}\leq C_k\tau^{-k}$ with $C_k$ independent of $\tau$, and where $w$ is some arbitrary initial vector and $c_0$ is some initial constant which determines the value of $a_{00}(t,r_0)$. Since it will simplify later calculations, we choose the value\begin{equation}\label{c_0}c_0=\frac{\sqrt[4]{\det\Imag(H(r_0))}}{\sqrt{\Big(\int_{\R^{m-1}}e^{-\abs{y}^2}dy\Big)}}=\frac{\sqrt[4]{\det\Imag(H(r_0))}}{\pi^{\frac{m-1}{4}}}.\end{equation}We note that the constructed $a_{00}$ is smooth in time, since the connection form $A$ is smooth in time and the transport equation for $a_{00}$ is well-posed. We can then obtain the rest of $a_{01},\dots,a_{0N}$ by solving linear first-order ODEs. The sections $a_1,\dots,a_N$ may then be determined in much the same manner as $a_0$, so that the transport equations in (\ref{transport}) are satisfied to $N$th degree on $\Gamma$, and the smooth time-dependence of the amplitude $a$ follows from the smooth time-dependence of the connection form and the well-posedness of the transport equations that determine $a$.\\

Thus, we have constructed a function $v=e^{is(\Psi-st)}a$ in $U$ such that:\begin{align*}&\Psi(r,y)=r+\frac{1}{2}H(r)y\cdot y+\tilde{\Psi},\ \textrm{where }\tilde{\Psi}=O(\abs{y}^3),\\&a(t,r,y)=s^{\frac{m-1}{4}}(a_0+s^{-1}a_1+\cdots+s^{-N}a_N)\chi(y/\delta'),\\&a_0(t,r,0)=c_0\tilde{\chi}(t)e^{-\frac{1}{2}\int_{r_0}^r\tr H(\tilde{r})d\tilde{r}}\ U_{A}w.\end{align*}We now turn to establishing the bounds (\ref{edit_p3est}) for the function $v$ in $U$. Henceforth, we shall choose $N=5$. To begin with, note that\[\abs{e^{is(\Psi-st)}}=e^{-s\Imag\Psi}=e^{-\frac{1}{2}s\Imag(H(r))y\cdot y}e^{-sO(\abs{y}^3)}.\]

Note also that $\Imag(H(r))y\cdot y\geq c\abs{y}^2$ for $(r,y)\in U$, where the constant $c>0$ depends on $H_0$ and the value of $\delta$ appearing in the definition of $U$. Thus, it follows for $(r,y)\in U$ that, provided $y$ is sufficiently small, we have\[\abs{e^{is(\Psi(r,y)-st)}}\leq Ce^{-\frac{1}{4}cs\abs{y}^2}.\]From this fact, together with the definition of $\chi(y/\delta')$ and $\tilde{\chi}(t)$, we can deduce that for $r$ in a compact interval, possibly after decreasing $\delta'$, we have\[\abs{v(t,r,y)}\leq Cs^{\frac{m-1}{4}}e^{-\frac{1}{4}cs\abs{y}^2}\chi(y/\delta'),\quad s\gg1.\]

As a result, we have that for $s\gg1$,\begin{equation}\label{bnd1}\norm{v}_{L^2((0,T)\times U;E)}\leq C\norm{s^{\frac{m-1}{4}}e^{-\frac{1}{4}cs\abs{y}^2}}_{L^2((0,T)\times U;E)}=O(1).\end{equation}

Further, since $\Psi,a$ satisfy (\ref{Eikonal}) and (\ref{transport}) to $5$th order on $\Gamma$, we deduce from (\ref{WKB}) that\[\abs{(i\d_t+\Delta_A+V)v}\leq Cs^{\frac{m-1}{4}}e^{-\frac{1}{4}cs\abs{y}^2}\chi(y/\delta')\Big(s^2\abs{y}^6+s^{-5}\Big).\]It follows, then, that\begin{equation}\label{bnd3}\norm{(i\d_t+\Delta_A+V)v}_{L^2((0,T)\times U;E)}\lesssim\norm{s^\frac{m-1}{4}e^{-\frac{1}{4}cs\abs{y}^2}\big(s^2\abs{y}^6+s^{-5}\big)}_{L^2((0,T)\times U;E)}=O\big(s^{-1}\big),\end{equation}and by the same method we may further deduce that\begin{equation}\label{bnd4}\norm{(i\d_t+\Delta_A+V)v}_{H^{1,0}((0,T)\times U;E)}\lesssim\norm{s^\frac{m-1}{4}e^{-\frac{1}{4}cs\abs{y}^2}\big(s^2\abs{y}^6+s^{-3}\big)}_{L^2((0,T)\times U;E)}=O\big(s\big).\end{equation}

Thus, we have established the result of Proposition \ref{edit_p3} for $v$ in $U$. It remains to show that we can construct an approximate Gaussian beam solution $v$ in $M$ by gluing together the approximate solutions which we have constructed in $U$.\\

To this end, recall the definition of the sets $U_j$ from Lemma \ref{Fermi}. Fix $\delta'$ and choose an open cover $U_0,\dots,U_{K}$ of $\gamma([a_0,b_0])$, with each $U_j$ corresponding to an interval $I_j$, as in Lemma \ref{Fermi}. We first find a function $v^{(0)}=e^{is(\Psi^{(0)}-st)}a^{(0)}$ in $U_0$ following the method above, with some fixed initial conditions at $r_0$ for the ODEs that determine $\Psi^{(0)}$ and $a^{(0)}$. We continue by choosing some $\tilde{r}_1\in I_0\cap I_1$ so that $\gamma(\tilde{r}_1)\in U_0\cap U_1$, and construct $v^{(1)}=e^{is(\Psi^{(1)}-st)}a^{(1)}$ in $U_1$ again by the above method, choosing our initial conditions for $\Psi^{(1)}$, $a^{(1)}$ at $\tilde{r}_1$ to be, respectively, the values of $\Psi^{(0)}$ and $a^{(0)}$ at $\tilde{r}_1$. In this manner, we can proceed to determine $v^{(K)}$. We choose a partition of unity $\{\rho_j(r)\}$ for $[a_0,b_0]$ corresponding to the family of intervals $\{I_j\}$, and let $\tilde{\rho}_j(r,y)=\rho_j(r)$ in $U_j$. We can then define\[v=\sum_{j=0}^K\tilde{\rho}_jv^{(j)}.\]

Since the ODEs for the phases and amplitudes have the same initial value in $U_j$ as in $U_{j+1}$, we can deduce that $v^{(j)}=v^{(j+1)}$ in $U_j\cap U_{j+1}$. Therefore, we conclude that the $L^2$-scale bounds (\ref{bnd1}) for $v$ and (\ref{bnd3})-(\ref{bnd4}) for $(i\d_t+\Delta_A+V)v$ follow with $U=M$ from the corresponding bounds in $U_l$ for each $v^{(l)}$.
\end{proof}

Before we proceed, we note also the following partition, which shall be useful later. Suppose that $p_1,\dots,p_{K'}$ are the distinct points where the geodesic self-intersects, $0<r_1<\cdots<r_K<L$ are the times where the geodesic self-intersects, and $V_1,\dots,V_{K'}$ are balls centered at $p_1,\dots,p_{K'}$ respectively. Note that in each $V_j$, the function $v$ is given by\[v\vert_{V_j}=\sum_{\gamma(t_l)=p_j}v^{(l)}.\] By considering the individual sets $U_l\setminus(\cup_{j=1}^{K'}V_j)$, we can then choose a finite cover $\{W_1,\cdots,W_{K''}\}$ of the remaining points of $(\cup_{l=1}^KU_l)\setminus(\cup_{j=1}^{K'}V_j)$, where each $W_k$ is a subset of $U_{l_k}$ for some $l_k$. In particular, we can choose $W_k$ which remain away from $\gamma(r_{l_k})$, so that for small enough $\delta'$ it follows that\[v\vert_{W_k}=v^{(l_k)}.\] Together, the $V_j$ and $W_k$ cover form a cover\begin{equation}\label{cover}\textrm{supp}(v)\cap M\subset \Big(\cup_{j=1}^{K'} V_j\Big)\cup\Big(\cup_{k=1}^{K''}W_k\Big).\end{equation}

\section{Determination of the Scattering Data}

Let us begin by establishing the following useful Lemma, which will aid us in the proof of Theorem \ref{t1}.

\begin{lemma}\label{edit_gaugefix}
Given any connection form $A$ in $M$, there is a gauge-equivalent connection form $\tilde{A}$ with the additional property that $\tilde{A}(\nu)\vert_{\d M}=0$.
\end{lemma}

\begin{proof}
Let us use boundary normal coordinates in a neighbourhood of $\d M$ (see e.g. \cite[Section 2.1]{C5}) given by $(y,r)\in\d M\times[0,\varepsilon]$, with $\varepsilon>0$. Then for $u\in C^\infty(M)$ it holds that $\d_ru\vert_{\d M}=-\d_\nu u\vert_{\d M}$. Further, in these coordinates, we can decompose the connection form as $A=A(\nu)dr+A(\d_{y^\alpha})dy^\alpha$. We fix some $\mu\in C^\infty([0,\varepsilon];\R)$ such that $\mu(r)=1$ for $r$ near $0$ and $\mu(r)=0$ for $r$ near $\varepsilon$. We may then define a local gauge $G\in C^\infty((0,T)\times\d M\times[0,\varepsilon];\C^{n\times n})$ via \[G(r,\cdot)=e^{-r\mu(r)A(\nu)\, (r,\cdot)},\] where the dependence on $t$ and $y$ is left implicit. Note that, since $A$ is skew-Hermitian, it follows that $G$ is unitary, and since $G(r)=\Id$ for $r$ near $\varepsilon$, we can extend $G$ to an element of $C^\infty((0,T)\times M;U(n))$. Moreover, observe that\[G\vert_{\d M}=\Id,\quad\d_rG\vert_{\d M}=-A(\nu)\vert_{\d M}.\] Let $\tilde{A}$ denote the gauge transform of $A$ by $G$, that is\[\tilde{A}=G^{-1}AG+G^{-1}dG\]and observe that\[\tilde{A}(\nu)\vert_{\d M}=A(\nu)\vert_{\d M}-A(\nu)\vert_{\d M}=0.\]\end{proof}

Since the data $\Lambda_{A_j,V_j}$ and $C_{A_j}$ are gauge invariant, Lemma \ref{edit_gaugefix} implies that we may assume without any loss of generality that $A_1(\nu)\vert_{\d M}=A_2(\nu)\vert_{\d M}=0$ in the proofs of Theorems \ref{t1} and \ref{t2}. With this in mind, we proceed to the proof of Theorem \ref{t1}.\\

For $j=1,2$, we construct Gaussian beam solutions $u_j$ of the Schr\"odinger equations\begin{equation}\begin{split}\big(i\d_t+\Delta_{A_j}+V_j\big)u_j(t,x)&=0\textrm{ in }(0,T)\times M,\\u_1(0,\cdot)=u_2(T,\cdot)&=0\textrm{ in }M.\end{split}\end{equation}To this end, we fix some geodesic $\gamma_{x,\theta}$ in $\widehat{M}$ for $x\in\d M$, and choose a system of Fermi coordinates along $\gamma_{x,\theta}$ as in Lemma \ref{Fermi}. Using the work of the previous section, we construct approximate solutions $v_j$ in $M$ having the form $e^{is(\Psi-st)}a^{(A_j)}$ away from the self-intersections of $\gamma_{x,\theta}$. We can turn these $v_j$ into exact solutions $u_j=v_j+R_j$ by solving\[\begin{split}\big(i\d_t+\Delta_{A_j}+V_j\big)R_j&=-(i\d t+\Delta_{A_j}+V_j)v_j\textrm{ in }(0,T)\times M,\\R_j&=0\textrm{ on }(0,T)\times\d M,\\R_1(0,\cdot)&=R_2(T,\cdot)=0\textrm{ in }M.\end{split}\]

Note that for $s\gg1$, (\ref{bnd3}) and the energy estimate (\ref{s2.3}) yield\begin{equation}\label{d1est}\norm{R_j}_{L^2((0,T)\times M;E)}\leq C\norm{(i\d_t+\Delta_{A_j}+V_j)v_j}_{L^2((0,T)\times M;E)}=O\big(s^{-1}\big),\end{equation}whereas (\ref{bnd4}) and the energy estimate (\ref{NRGest}) yield\[\norm{R_j}_{H^{0,2}((0,T)\times M;E)}\leq C\norm{(i\d_t+\Delta_{A_j}+V_j)v_j}_{H^{1,0}((0,T)\times M;E)}=O(s).\]

Interpolating between (\ref{d1est}) and the above, we conclude that for $s\gg1$ we have\begin{equation}\label{Rjbnd}\norm{R_j}_{H^{0,1}((0,T)\times M;E)}+s\norm{R_j}_{L^2((0,T)\times M;E)}=O(1).\end{equation}

We then set $\phi_j=u_j$ on $(0,T)\times\d M$ and consider $\omega\in H^{1,2}((0,T)\times M;E)$ the solution of
\begin{equation}\begin{split}\label{3.1}(i\d_t+\Delta_{A_2(t)}+V_2(t,x))\omega(t,x)&=0\textrm{ in }(0,T)\times M,\\\omega(t,x)&=\phi_1\textrm{ on }(0,T)\times\d M,\\\omega(0,\cdot)&=0\textrm{ in }M.\end{split}\end{equation}

We observe that the difference $\omega-u_1$ solves the following Schr\"odinger equation:

\begin{equation}\begin{split}\label{3.2}\big(i\d_t+\Delta_{A_2}+V_2\big)(\omega-u_1)&=2(A_1-A_2,du_1)_g+Qu_1\textrm{ in }(0,T)\times M,\\\omega-u_1&=0\textrm{ on }(0,T)\times\d M,\\\omega(0,x)-u_1(0,x)&=0\textrm{ in }M,\end{split}\end{equation}
where $Qu_1=(V_1-V_2)u_1+(A_1,A_1u_1)_g-(A_2,A_2u_1)_g-(d^\ast A_1)u_1+(d^\ast A_2)u_1$.\\

Taking the Hermitian inner product of the above equation with $u_2$, we deduce that \begin{equation}\label{3.3}\int_0^T\int_M\pair{2(A_1-A_2,du_1)_g+Qu_1,u_2}_EdV_gdt=\int_0^T\int_{\d M}\pair{\d_\nu(\omega-u_1),u_2}_Ed\sigma_gdt.\end{equation}

As a result of Lemma \ref{edit_gaugefix}, we may assume without loss of generality that $A_1(\nu)\vert_{\d M}=A_2(\nu)\vert_{\d M}=0$. Therefore, the right-hand side of the above is bounded by \begin{equation}\begin{split}\label{3.4}\abs{\int_0^T\int_{\d M}\pair{\d_\nu(\omega-u_1),u_2}_Ed\sigma_gdt}&\leq C\norm{(\Lambda_{A_1,V_1}-\Lambda_{A_2,V_2})\phi_1}_{L^2((0,T)\times\d M;E)}\norm{\phi_2}_{L^2((0,T)\times\d M;E)}\\&\leq C\norm{\Lambda_{A_1,V_1}-\Lambda_{A_2,V_2}}\norm{\phi_1}_{H^{\frac{9}{4},\frac{3}{2}}((0,T)\times\d M;E)}\norm{\phi_2}_{L^2((0,T)\times\d M;E)},\end{split}\end{equation}which vanishes when $\Lambda_{A_1,V_1}=\Lambda_{A_2,V_2}$. On the other hand, since $du_1=is(d\Psi)u_1+e^{is(\Psi-st)}da^{(A_1)}+dR_1$, the left-hand side of (\ref{3.3}) can be written as

\[\begin{split}\int_0^T\int_M\pair{2(A_1-A_2,du_1)_g+Qu_1,u_2}_EdV_gdt=is\int_{(0,T)\times M}\pair{2\big((A_1-A_2),(d\Psi)u_1\big)_g,u_2}_EdV_gdt\\+\int_{(0,T)\times M}\pair{2\big((A_1-A_2),e^{is(\Psi-st)}da^{(A_1)}+dR_1\big)_g+Qu_1,u_2}_EdV_gdt.\end{split}\]

We can divide the above by $s$ and use the bounds (\ref{Rjbnd}) and (\ref{bnd1}) to deduce that\begin{equation}\begin{split}\label{3.5}&\abs{\int_{(0,T)\times M}\pair{\big((A_1-A_2),(d\Psi)u_1\big)_g,u_2}_EdV_gdt}\\\leq &s^{-1}\abs{\int_{(0,T)\times M}\pair{\big((A_1-A_2),du_1\big)_g+Qu_1,u_2}_EdV_gdt}+O\big(s^{-1}\big).\end{split}\end{equation}

By combining (\ref{3.3}), (\ref{3.4}), and (\ref{3.5}), we conclude that when $\Lambda_{A_1,V_1}=\Lambda_{A_2,V_2}$ it follows that\[\abs{\int_{(0,T)\times M}\pair{\Big((A_1-A_2),(d\Psi)u_1\Big)_g\,,u_2}_EdV_gdt}\leq O(s^{-1}).\]Thus, we can let $s\rightarrow\infty$ in the right-hand side of the above to conclude that \begin{equation}\label{3.6}\lim_{s\rightarrow\infty}\int_{(0,T)\times M}\pair{\Big((A_1-A_2),(d\Psi)u_1\Big)_g\,,u_2}_EdV_gdt=0.\end{equation}

We now make use of the following Lemma:

\begin{lemma}\label{SPC}
\[\begin{split}&\lim_{s\rightarrow\infty}\int_{(0,T)\times M}\pair{\Big((A_1-A_2),(d\Psi)u_1\Big)_g\,,u_2}_EdV_gdt\\=\int_0^T&\tilde{\chi}^2\int_{0}^{\rho_+(x,\theta)}\pair{\Big(A_1\big(\gamma'_{x,\theta}(r)\big)-A_2\big(\gamma'_{x,\theta}(r)\big)\Big)U_{A_1}w_1,U_{A_2}w_2}_Edrdt.\end{split}\]
\end{lemma}

Before proving the above result, we first conclude the proof of Theorem \ref{t1}. Since $\tau\in(0,T/4)$ in the definition of $\tilde{\chi}$ is arbitrary, we deduce that\begin{equation}\label{3.8}\int_{0}^{\rho_+(y,\theta)}\pair{\Big(A_1\big(\gamma'_{x,\theta}(r)\big)-A_2\big(\gamma'_{x,\theta}(r)\big)\Big)U_{A_1}w_1,U_{A_2}w_2}_Edr=0.\end{equation}

Note that the scattering data $C_{A_j}$ takes values in $U(n)$, so that we can define $C_{A_j}^{-1}$ as the matrix inverse of $C_{A_j}$. Making use of (\ref{HIP}), (\ref{PT2}), and (\ref{ScatteringData}), it can be shown that\begin{equation}\begin{split}\label{3.9}\pair{\big(C_{A_2}^{-1}C_{A_1}-\Id\big)w_1,w_2}_E=&\int_{0}^{\rho_+(y,\theta)}\d_r\pair{U_{A_1}w_1,U_{A_2}w_2}_Edr\\=&\int_{0}^{\rho_+(y,\theta)}\pair{\nabla_{\gamma'_{x,\theta}}^{A_2}U_{A_1}w_1,U_{A_2}w_2}_Edr+\int_{0}^{\rho_+(y,\theta)}\pair{U_{A_1}w_1,\nabla_{\gamma'_{x,\theta}}^{A_2}U_{A_2}w_2}_Edr\\=&\int_{0}^{\rho_+(y,\theta)}\pair{\Big(A_2\big(\gamma'_{x,\theta}(r)\big)-A_1\big(\gamma'_{x,\theta}(r)\big)\Big)U_{A_1}w_1,U_{A_2}w_2}_Edr.\end{split}\end{equation}

However, (\ref{3.8}) and (\ref{3.9}) imply that $\pair{(C_{A_2}^{-1}C_{A_1}-\Id)w_1,w_2}_E=0$. Since $w_1,w_2$ were arbitrary, we conclude that $C_{A_2}^{-1}C_{A_1}=\Id$, whence $C_{A_1}=C_{A_2}$ and the proof of Theorem \ref{t1} is complete. It remains for us to provide a proof of Lemma \ref{SPC}.

\begin{proof}[Proof of Lemma \ref{SPC}]
By considering a partition of unity subordinate to the open cover (\ref{cover}), it is sufficient to consider $A_1-A_2$ compactly supported in $(V_j\cap M)$ or $(W_k\cap M)$. We shall first prove the latter case. Recall that $a^{(A_j)}=s^{\frac{m-1}{4}}(a_0+O(s^{-1}))\chi(y/\delta')$. We can observe that
\[\begin{split}\lim_{s\rightarrow\infty}&\int_{(0,T)\times M}\pair{\Big((A_1-A_2),(d\Psi)u_1\Big)_g\,,u_2}_EdV_gdt\\=\lim_{s\rightarrow\infty}&\int_0^T\int_{0}^{\rho_+(x,\theta)}\int_{\R^{m-1}}e^{-s\Imag(H(r))y\cdot y}e^{sO(\abs{y}^3)}s^{\frac{m-1}{2}}\chi^2(y/\delta')\times\\&\Big[\pair{\Big(A_1(t,r,y)-A_2(t,r,y),d\Psi(t,r,y) a_0^{(A_1)}(t,r,y)\Big)_g\,,a_0^{(A_2)}(t,r,y)}_E+O(s^{-1})\Big]\abs{g(r,y)}^\frac{1}{2}dydrdt.\end{split}\]

We make the substitution $y\mapsto s^{-\frac{1}{2}}y$ above, and recall that $d\Psi\vert_{y=0}=dr$ is dual via the Riemannian metric to $\gamma'_{x,\theta}(r)$. Then, since $\Imag(H)$ is positive definite and $\delta'$ is small, we note that the exponential term involving $\Imag (H)$ dominates the others. Thus, we can conclude that the right-hand side of the above is given by\[\int_0^T\int_{0}^{\rho_+(x,\theta)}\Big(\int_{\R^{m-1}}e^{-\Imag(H(r))y\cdot y}dy\Big)\pair{\big(A_1(\gamma'_{x,\theta})-A_2(\gamma'_{x,\theta})\big)a_0^{(A_1)}(t,r,0),a_0^{(A_2)}(t,r,0)}_E\abs{g(r,0)}^{\frac{1}{2}}drdt.\]

Evaluating the integral in $y$ and using that $\abs{g(r,0)}=1$, we rewrite the above integral in the form

\[\int_0^T\tilde{\chi}^2\Big(\int_{\R^{m-1}}e^{-\abs{y}^2}dy\Big)\int_{0}^{\rho_+(x,\theta)}\frac{\abs{c_0}^2e^{-\int_{r_0}^r\tr\Real(H(s))ds}}{\sqrt{\det\Imag(H(r))}}\pair{\Big(A_1(\gamma'_{x,\theta})-A_2(\gamma'_{x,\theta})\Big)U_{A_1}w_1,U_{A_2}w_2}_Edrdt.\]

We now use the fact that, according to \cite[Lemma 2.58]{C5}, matrices which solve the Riccati equation (\ref{Riccati}) have the property\[\det\Imag(H(r))=\det\Imag(H(r_0))e^{-2\int_{r_0}^r\tr\Real(H(s))ds}.\]

This fact, together with our choice of constant $c_0$ in (\ref{c_0}) is sufficient to conclude that\[\begin{split}\lim_{s\rightarrow\infty}&\int_{(0,T)\times M}\pair{\big((A_1-A_2),(d\Psi)u_1\big)_g\,,u_2}_EdV_gdt\\=\int_0^T\tilde{\chi}^2&\int_{0}^{\rho_+(y,\theta)}\pair{\Big(A_1\big(\gamma'_{x,\theta}(r)\big)-A_2\big(\gamma'_{x,\theta}(r)\big)\Big)U_{A_1}w_1,U_{A_2}w_2}_Edrdt,\end{split}\]when $A_1-A_2$ is compactly supported in $W_k\cap M$. For $A_1-A_2$ compactly supported in $V_j\cap M$, we can write $v=\sum_{\gamma(t_l)=p_j}v^{(l)}$. Thus, the limit has the form\[\begin{split}\lim_{s\rightarrow\infty}\int_{(0,T)\times M}\pair{\big((A_1-A_2),(d\Psi)u_1\big)_g\,,u_2}_EdV_gdt\\=\sum_{\gamma(t_l)=p_j}\lim_{s\rightarrow\infty}\int_{(0,T)\times M}\pair{\big((A_1-A_2),(d\Psi)v_1^{(l)}\big)_g\,,v_2^{(l)}}_EdV_gdt\\+\sum_{\substack{l\neq l'\\{\gamma(t_l)=\gamma(t_{l'})=p_j}}}\lim_{s\rightarrow\infty}\int_{(0,T)\times M}\pair{\big((A_1-A_2),(d\Psi)v_1^{(l)}\big)_g\,,v_2^{(l')}}_EdV_gdt.\end{split}\]

We observe that the first sum converges to the required limit by the same compuations used to prove the limit in  $W_k\cap M$, and the second sum vanishes via stationary phase arguments as in the proof of \cite[Proposition 3.1]{C4}, thus completing the proof of Lemma \ref{SPC}.\end{proof}

\section{Proof of Gauge Equivalence}

In what follows, we denote by $\varphi_r(x,v)$ the geodesic flow given by $\varphi_r(x,v)=(\gamma_{x,v}(r),\gamma'_{x,v}(r))\in TM$. Additionally, we denote by $X$ the geodesic vector field on $M$, which satisfies $\d_r(F\circ\varphi_r)=X(F)\circ\varphi_r$, for any function $F:SM\mapsto\C^n$.\\

For a function $F:SM\rightarrow\C^n$ and a connection $1$-form $B:TM\mapsto\C^{n\times n}$, we can consider the transport equation:
\begin{equation}\begin{split}\label{ray}Xw+Bw=&-F\textrm{ in }SM\\w\vert_{\d_-SM}=&0,\end{split}\end{equation}where $B$ acts on $w:SM\rightarrow\C^n$ by multiplication at each point. In the present work, we shall consider only the cases $F(x,\theta)=f(x)$ or $F(x,\theta)=\alpha_j(x)\theta^j$, where $f,\alpha_j:M\rightarrow\C^n$.  Using the transport equation (\ref{ray}), we can define the attenuated ray transform of $F$ with attenuation due to $B$ (see e.g. \cite[Section 1]{C1}), via:\begin{equation}\label{ARTDefn}I_BF=w\vert_{\d_+SM}.\end{equation}

We can now proceed to the proof of Theorem \ref{t2} below.

\begin{proof}[Proof of Theorem \ref{t2}]
In order to prove Theorem \ref{t2}, we further assume that $M$ is either i) $2$-dimensional and simple, or ii) of dimension $m\geq3$ with strictly convex boundary and admits the existence of a smooth strictly convex function $\phi:M\rightarrow\R$. We note that this second condition is conjectured to be true for all simple manifolds (amongst others, see e.g. \cite[Section 2]{C1} for discussion), although the question is still open at present.\\

Consider the candidate gauge $G(t)=U_{A_1(t)}(U_{A_2(t)})^{-1}$. Note that $G(t)$ is unitary, $G(t):SM\rightarrow U(n)$, and depends smoothly on time. Further, since $C_{A_1}=C_{A_2}$, it holds that $G(t)\vert_{\d_+SM}=\Id$, and therefore that $G(t)\vert_{\d SM}=\Id$. We can observe that\begin{equation}\begin{split}\label{6.1}XG+A_1G-GA_2&=0\\G\vert_{\d SM}&=\Id.\end{split}\end{equation}

It remains to show that $G(t)$ depends only on the base-point $x\in M$, and further that $G(t)$ is smooth in $M$. Note that (\ref{6.1}) is equivalent to\begin{equation}\begin{split}\label{6.2}X(G-\Id)+A_1(G-\Id)-(G-\Id)A_2=A_2-A_1\\(G-\Id)\vert_{\d SM}=0.\end{split}\end{equation}

We can henceforth fix some $t\in(0,T)$ and define a new connection form $B$ via $B(W)=A_1W-WA_2$ for $W:SM\rightarrow\C^{n\times n}$. Then (\ref{6.2}) becomes\begin{equation}\begin{split}\label{ARTGauge}X(G-\Id)+B(G-\Id)=A_2-A_1\\(G-\Id)\vert_{\d SM}=0.\end{split}\end{equation}

We can interpret the above equation as a ray transform (see e.g. \cite[Section 1]{C1}), as $I_B(A_1-A_2)=0$, where $I_B$ denotes the ray transform with attenuation due to the connection form $B$. We emphasize here that $A_1$, $A_2$ and $B$ are smooth, whereas we have not yet shown that $G(t)$ in (\ref{ARTGauge}) is smooth.\\

We now apply either \cite[Theorem 1.3]{C3} if $M$ is $2$-dimensional and simple, or \cite[Theorem 1.6]{C1} if $M$ is of dimension $m\geq3$ with strictly convex boundary and admits a smooth strictly convex function. Thus, we conclude that $A_2(t)-A_1(t)=\nabla^{B(t)}p(t)$ for some smooth function $p(t):M\rightarrow \C^{n\times n}$, and (\ref{ARTGauge}) then implies that $G(t)=\Id+p(t)$. Hence, we have shown that $G(t)$ depends only on the base-point $x\in M$, and the smoothness of $G(t)$ in $M$ follows from the smoothness of $p(t)$. Therefore, $G$ satisfies all the necessary conditions to be a gauge, and it follows from (\ref{6.1}) that $A_1$ is gauge-equivalent to $A_2$ via the gauge-transform $A_2=G^{-1}dG+G^{-1}A_1G$, as required.\\

We now turn to showing unique determination of the potential. Note that since $A_2=G^{-1}dG+G^{-1}A_1G$, it remains only to show that $V_2=G^{-1}V_1G+iG^{-1}\d_tG$. We define $V_3=G^{-1}V_1G+iG^{-1}\d_tG$. By gauge invariance, it holds that $\Lambda_{A_1,V_1}=\Lambda_{A_2, V_3}$. Thus, by the assumption $\Lambda_{A_1,V_1}=\Lambda_{A_2,V_2}$ it follows that\begin{equation}\label{6.3}\Lambda_{A_2,V_3}=\Lambda_{A_2,V_2}.\end{equation}

It remains to show that condition (\ref{6.3}) implies that $V_2=V_3$. Note that we can take $A_1=A_2=A$ in (\ref{3.3}), and use the fact that $\Lambda_{A,V_3}=\Lambda_{A,V_2}$ to deduce that \[\int_{(0,T)\times M}\pair{(V_3-V_2)u_1,u_2}_E\circ\varphi_r(y,\theta)dV_gdt=0.\] We can let $s\rightarrow\infty$ in the left-hand side above, and applying the argument of Lemma \ref{SPC} we deduce that\[\int_0^T\tilde{\chi}^2(t)\int_0^{\rho_+(y,\theta)}\pair{(V_3-V_2)U_Aw_1,U_Aw_2}_E\circ\varphi_r(y,\theta)drdt=0.\] Since the choice of $\tilde{\chi}$ is arbitrary, we conclude that for $V=V_3-V_2$ and for each $t\in(0,T)$ and $(y,\theta)\in\d_+SM$ we have\[\int_{0}^{\rho_+(y,\theta)}\pair{VU_Aw_1,U_Aw_2}_E\circ\varphi_r(y,\theta)dr=0,\] where $\varphi_r(y,\theta)$ is the geodesic flow defined by $\varphi_r(y,\theta)=(\gamma_{y,\theta}(r),\gamma'_{y,\theta}(r))$. Then, by linearity, we conclude that\begin{equation}\label{6.4}\int_{0}^{\rho_+(y,\theta)}{U_A}^{-1}VU_A\circ\varphi_r(y,\theta)dr=0.\end{equation}

In order to finish the proof, we wish to interpret (\ref{6.4}) as an attenuated ray transform.\\

For $W\in C^\infty(SM;\C^{n\times n})$, we can define the map $BW=AW-WA=[A,W]$. Since $A(x,v)=A_j(x)v^j$, we can observe that $BW(x,v)=[A_j(x),W]v^j$.\\

Recall that the inner product $\pair{\cdot,\cdot}_E$ on the trivial bundle $E=M\times\C^n$ induces an inner product on the endomorphism bundle $\End(E)=M\times\C^{n\times n}$ via $\pair{X,Y}_{\End(E)}=\tr(X^\ast Y)$. Note that we can regard the $1$-form $B$ as a connection form on $\End(E)$. Further, by the cyclic property of trace, we observe that \[\pair{BX,Y}_{\End(E)}=\tr(-X^\ast AY+AX^\ast Y)=\tr(-X^\ast AY+X^\ast YA)=\pair{X,-BY}_{\End(E)},\] and that $B$ is, therefore, a unitary connection on the endomorphism bundle.\\

Letting $X$ once again denote the geodesic vector field, we consider $\omega$ the solution of the transport equation
\begin{equation}\label{6.5}X\omega+B\omega=-V\textrm{ in }SM,\quad \omega=0\textrm{ on }\d_-SM.\end{equation}Using (\ref{6.5}), we now observe that \[X({U_A}^{-1}\omega U_A)={U_A}^{-1}A\omega U_A+{U_A}^{-1}(-V-B\omega)U_A-{U_A}^{-1}\omega AU_A=-{U_A}^{-1}VU_A.\]

The above then implies that $\d_r[({U_A}^{-1}\omega U_A)\circ\varphi_r(y,\theta)]=-{U_A}^{-1}VU_A\circ\varphi_r(y,\theta)$, using the definitions of the geodesic vector field and geodesic flow. We can integrate this last expression to obtain\begin{equation}\label{6.6}-\int_{0}^{\rho_+(y,\theta)}{U_A}^{-1}VU_A\circ\varphi_r(y,\theta)dr=({U_A}^{-1}\omega U_A)\circ\varphi_{\rho_+(y,\theta)}(y,\theta)-({U_A}^{-1}\omega U_A)\circ\varphi_{0}(y,\theta).\end{equation}

Note that $\varphi_{\rho_+(y,\theta)}(y,\theta)\in\d_-SM$ and $\varphi_{0}(y,\theta)=(y,\theta)\in\d_+SM$. By recalling that that $U_A=\Id$ on $\d_+SM$ and $\omega=0$ on $\d_-SM$, we observe that the right-hand side of (\ref{6.6}) is just $-\omega\vert_{\d_+SM}$.\\

Therefore, (\ref{6.4}) and (\ref{6.5}) tell us that\[I_BV=\omega\vert_{\d_+SM}=\int_{0}^{\rho_+(y,\theta)}{U_A}^{-1}VU_A\circ\varphi_r(y,\theta)dr=0.\]

Hence, if $M$ is $2$-dimensional and simple, we can apply \cite[Theorem 1.3]{C3} to conclude that the above implies $V=0$. On the other hand, if $M$ is of dimension $m\geq3$ with strictly convex boundary and admits a smooth strictly convex function, we instead apply the result of \cite[Theorem 1.6]{C1} to conclude that $V=0$. Thus, it holds that $V_2=G^{-1}V_1G+iG^{-1}\d_tG$, and hence $(A_1,V_1)$ is gauge-equivalent to $(A_2,V_2)$, as required.
\end{proof}

\section*{Acknowledgments}
AT was supported by EPSRC DTP studentship EP/N509577/1.

\begin{small}

\end{small}

\end{document}